\documentclass[11pt]{article}
\usepackage{amsmath}
\usepackage{amsthm}
\usepackage{amssymb}
\usepackage{graphicx, epsfig}
\usepackage{color}
\usepackage{array}

\def\a{\mu}
\def\b{\beta}
\def\g{\gamma}
\def\r{\rho}

\def\om{\overline{\mu}}
\def\ov{\overline{v}}

\newtheorem{theorem}{Theorem}
\newtheorem{lemma}[theorem]{Lemma}

\newtheorem{claim}[theorem]{Claim}

\newcommand{\fee}{{\mathcal F}}

\newcommand{\cb}{{\mathcal B}}
\newcommand{\cn}{{\mathcal N}}
\newcommand{\cm}{{\mathcal M}}
\newcommand{\cd}{{\mathcal D}}
\newcommand{\cx}{{\mathcal X}}
\newcommand{\cy}{{\mathcal Y}}
\newcommand{\cz}{{\mathcal Z}}

\newcommand{\cp}{{\mathcal P}}
\newcommand{\e}{\epsilon}

\def\t{\tau}

\begin{document}

\title{The Triangle-Free Process}

\author{Tom Bohman \thanks{Department of Mathematical Sciences,
Carnegie Mellon University, Pittsburgh, PA, USA. Supported in part by
NSF grant DMS 0701183. E-mail: \texttt{tbohman@math.cmu.edu}}
}

\maketitle

\begin{abstract}
\noindent
Consider the following stochastic graph process.  We begin with \(G_0\), 
the empty 
graph on \(n\) vertices, and form  \( G_{i} \) by adding a randomly chosen edge \(e_i\) to 
\( G_{i-1} \) where \( e_i\) is chosen
uniformly at random from the collection of pairs of vertices that 
neither appear as edges in \(G_{i-1}\) nor form triangles when 
added as edges to \(G_{i-1}\).  Let the random variable \(M\) be 
the number of edges in the
maximal triangle free graph generated by this process.  We prove
that asymptotically almost surely 
\( M = \Theta (n^{3/2} \sqrt{ \log n} ) \).  This resolves a conjecture
of Spencer.  Furthermore, the
independence number of \( G_M\) is asymptotically almost surely 
\( \Theta(\sqrt{ n \log n} ) \), which implies that the Ramsey number
\( R(3,t) \) is bounded below by a constant times \( t^2/ \log t \) (a fact that was previously
established by Jeong Han Kim).  The methods introduced here extend to the \(K_4\)-free
process, thereby establishing the bound \( R(4,t) = \Omega( t^{5/2}/ \log^2 t ) \).
\end{abstract}

\section{Introduction}

Consider the following constrained random graph process.
We begin with the empty graph on \(n\) vertices, which we denote \( G_0 \).
At step \(i\) we form the graph \( G_{i} \) by adding an edge to \( G_{i-1} \) chosen uniformly
at random from the collection of pairs of vertices that neither appear as edges in \( G_{i-1} \)
nor form triangles (i.e. copies of \(K_3\)) when added as edges to \(G_{i-1}\).  The process
terminates with a maximal triangle-free graph on \(n\) vertices, which we denote \( G_M \) (thus the
random variable \(M\) is the number of steps in the process).
We are interested in the likely
structural properties of \(G_M\) as \(n\) tends to infinity; for example, we 
would like to know the value of \(M\) and the independence number of \(G_M\).  

The study of this
graph processes began by the late 1980's (see Bollob\'as \cite{bela!}). 
The first published result on a process 
which iteratively adds edges chosen uniformly at random 
from the collection of potential edges that maintain some graph property 
is
due to Ruci\'nski and Wormald, who answered a question of Erd\H{o}s regarding 
the process in which we maintain
a bound on the maximum degree \cite{rw}.  Erd\H{o}s, Suen and Winkler considered both the triangle-free
process and the odd-cycle-free process 
\cite{esw}.  
The \(H\)-free process where \(H\) is a fixed graph was
treated by Bollob\'as and Riordan \cite{br} as well as Osthus and Taraz \cite{ot}.  While these papers
establish interesting bounds on the likely number edges in the graph produced by the 
\(H\)-free process for some
graphs \(H\), for no graph \(H\) that contains a cycle has the exact order of
magnitude 
been determined.

Another motivation for the triangle-free process comes 
from Ramsey theory.  The Ramsey number \( R(k,\ell) \) is the minimum integer \(n\) such that
any graph on \(n\) vertices contains a clique on \(k\) vertices or an independent set on
\( \ell \) vertices.  The Ramsey numbers play a central role in 
combinatorics and are the subject of many notoriously 
difficult problems, most of which remain widely open (see, for example, \cite{grs} \cite{cg}).  
One problem regarding 
the Ramsey numbers that
has been resolved is the order of magnitude of \( R(3,t) \) 
as \(t\) tends to infinity: Ajtai, Koml\'os and Szemer\'edi proved the upper bound 
\( R(3,t) = O( t^2/ \log t ) \) and
Kim established the lower bound \( R(3,t) = \Omega( t^2/ \log t) \).  (There were a number of
significant steps over the course of about 30 years that led up to these final results; see 
\cite{e1}, \cite{el}, \cite{e2},  \cite{e3}, \cite{e4}, 
\cite{esw}, \cite{e5}, \cite{s}.)  
The problem of determining the asymptotic behavior of \( R(3,t) \)
was one of the motivations
for the introduction of the triangle-free process.  Indeed, we establish a lower bound of the form
\( R(3,t) > n \) by proving the existence of a graph on \(n\) vertices 
with neither a triangle nor an independent set on \(t\) vertices, and 
the triangle-free process should produce such a graph as it should include enough `random' edges
to eliminate all large independent sets.  In a certain sense, Kim's celebrated result verified 
this intuition as he used a semi-random variation on the triangle-free process.  (This was
an application of the powerful R\"odl nibble that was inspired by an approach to the
triangle-free process proposed by Spencer \cite{s}.)
However, the problem of whether or not the triangle-free process itself is likely to
produce a Ramsey \( R(3,t) \) graph remained open.

Our main results (Theorems~\ref{thm:lower}~and~\ref{thm:upper} below) have the following Corollaries.
\begin{theorem}
\label{thm:show1}
Let the random variable \(M\) be the number of edges in the graph on \(n\) 
vertices formed by the triangle-free
process. There are constants \( c_1, c_2 \) such that asymptotically almost surely we have
\[ c_1 \sqrt{ \log n} \cdot n^{3/2} \le M \le  c_2 \sqrt{ \log n} \cdot n^{3/2}. \]
\end{theorem}
\noindent
\begin{theorem}  
\label{thm:show2}
There is a constant \( c_3 \) such that the following holds:  
If \( n = n(t) < c_3 \cdot t^2/ \log t \) then a.a.s. the triangle-free process on \(n\) vertices 
produces a graph with no independent set of cardinality \(t\).
Thus
\( R(3,t) \ge  c_3 \cdot \frac{ t^2 }{ \log t} \) for \(t\) sufficiently large.
\end{theorem}
\noindent
Theorem~\ref{thm:show1}  proves a conjecture of Spencer \cite{s}.  Theorem~\ref{thm:show2}, which 
establishes that the triangle-free process is an effective randomized algorithm for producing
a Ramsey \( R(3,t) \) graph, is a direct consequence of 
Theorem~\ref{thm:upper} below.  These results reveal a visionary aspect of 
the 1961 paper of Erd\H{o}s \cite{e1} which
established the bound \( R(3,t) = \Omega( t^2/ \log^2 t ) \).
When the probabilistic method was in its infancy, Erd\H{o}s established his bound by analyzing a
greedy algorithm applied to a random graph, and it turns out that a random greedy algorithm produces a 
Ramsey \( R(3,t) \) graph.
For an explicit construction of a triangle-free graph on \( \Theta( t^{3/2}) \) vertices
with independence number \(t\) see Alon \cite{a}.

The methods introduced here can be 
applied to other processes.  In fact, our methods immediately 
suggest an approach to the \(H\)-free
process for general \(H\).  This applies to hypergraph 
processes as well.  
As an example, we analyze the \( K_4 \)-free process to prove the
following result:
\begin{theorem}
\label{thm:extra}  There is a constant \( c_4\) such that for \(t \) sufficiently large we have
\[ R(4,t) > c_4 \cdot  \frac{ t^{5/2} }{ \log^2 t}. \]
\end{theorem}
\noindent
This is a minor improvement on the previously best known lower bound,
\( \Omega \left( (t/ \log t)^{5/2} \right) \), which was established
by Spencer via an application of the Lov\'asz Local Lemma \cite{e3}. 

We analyze the triangle-free process by an application of the so-called differential
equations method for random graph processes (see Wormald \cite{w} for an 
introduction to the method).  The main idea is to identify a collection of random
variables whose one-step expected changes can be written in terms of the 
random variables in the
collection.  These expressions yield an autonomous system of ordinary differential
equations, and we prove that the random variables in our collection (appropriately scaled) 
are tightly concentrated around the trajectory given by the solution of the ode.
Recent applications of this method include results that link the emergence of a giant 
component in a random graph process to a blow-up point in an associated ode
\cite{bk}, \cite{sw}, \cite{bbfp} and an analysis of a randomized matching algorithm
that hinges on the existence of an invariant set in an associated ode \cite{bf}.

We track the following random variables through the evolution 
of the triangle-free process.   Recall that
\( G_i\) is the graph given by the first \(i\) edges selected by the process.
The graph \( G_i \) partitions \( \binom{ [n]}{2} \) into three parts: \( E_i, O_i \) and
\(C_i\).  The set \( E_i\) is simply the edge set of \( G_i\).  A pair \( \{u,v\} \in \binom{[n]}{2} \) is {\bf open}, and 
in the set \( O_i \), if it can still be added as an edge without violating the triangle-free
condition.  A pair \( \{u,v\} \in \binom{ [n]}{2} \) is {\bf closed}, and in the set \( C_i\), if it is neither an 
edge in the graph nor open; that is, the pair \( \{u,v\} \) is in  \( C_i \) if 
there some vertex \(w\) such that \( \{u,w\}, \{v,w\} \in E_i \).  Note that \( e_{i+1} \) is chosen
uniformly at random from \( O_i\).  Set
\( Q(i) = | O_i | \); this is one of the random variables we track.
For 
each pair \( \{u,v\} \in \binom{[n]}{2} \) we track three random variables.  Let
\( X_{u,v} (i) \) be the set of vertices \(w\) such that 
\( \{u,w\}, \{v,w\} \in O_i \).  Let \( Y_{u,v}(i) \) be the set of vertices \(w\) such that
\[ \left| \left\{ \{ u,w\} , \{v,w\} \right\} \cap O_i \right| =
 \left| \left\{ \{ u,w\} , \{v,w\} \right\} \cap E_i \right| = 1.\]
Finally, let \( Z_{u,v}(i) \) be the set of vertices \(w\) such that
\( \{ u,w\}, \{v,w \} \in E_i \).  Note that if \( Z_{u,v}(i) \neq \emptyset \) then
we have \( \{ u,v\} \in C_i \).  
We dub vertices in \( X_{u,v} \) {\bf open} with respect to
\( \{u, v\} \), vertices in \( Y_{u,v} \) {\bf partial} with respect to \( \{u,v\} \)
and vertices
in \( Z_{u,v} \) {\bf complete} with respect to \( \{u,v\} \).  
We track the 
variables \( \left| X_{u,v} (i) \right|\), \( \left|Y_{u,v} (i) \right| \) and 
\( \left| Z_{u,v}(i) \right| \) for all 
pairs \(u,v \) such that
\( \{ u,v \} \not\in E_i \).  (In fact, we only show that \( |Z_{u,v}(i) | \) does not
get too large; so we track this random variable in the sense that we bound it).
We emphasize that we make no claims 
regarding the number of open, partial and complete
vertices with respect to pairs \(\{ u,v \} \) that are edges in the graph.  Formally, we set
\( X_{u,v}(i) = X_{u,v}(i-1) \), \( Y_{u,v}(i) = Y_{u,v}(i-1) \) and 
\( Z_{u,v}(i) = Z_{u,v}(i-1) \) if 
\( \{u,v\} \in E_i \).

In order to motivate our main results (Theorems~\ref{thm:lower}~and~\ref{thm:upper} below) we 
present a heuristic derivation of the trajectory that the random 
variables \( Q(i), \left| X_{u,v}(i) \right| \) and \( \left| Y_{u,v}(i) \right| \) should follow.  We stress that this
discussion does not constitute a proof that the random variables follow this trajectory; the
proof itself comes in Section~\ref{sec:lower} below.
We begin by choosing appropriate scaling.  We introduce a continuous variable \(t\) and relate 
this to the steps \(G_i\) in the process by setting 
\( t = t(i) = i/ n^{3/2} \).  Our trajectories are given by three functions: 
\(q(t), x(t)\) and \( y(t) \). We suppose \( Q(i) \) is approximately \( q(t) n^2 \),
\( |X_{u,v}(i)| \) is approximately \( x(t) n \) for all 
\( \{u,v\} \in \binom{[n]}{2} \setminus E_i \)
and
\( |Y_{u,v}(i)| \) is approximately \( y(t) \sqrt{n} \) for all 
\( \{u,v\} \in \binom{[n]}{2} \setminus E_i \).  Consider a fixed step \(i\) in the graph process and let 
\( \epsilon > 0 \) be sufficiently small. 
We suspect that the changes in our tracked 
random variables are very close to their expected values over the ensuing \( \epsilon n^{3/2} \)
steps of the process and use this guess to derive our system of differential equations.
We begin with \( |O_i| \).
Note that if \( e_{i+1} = \{u,v\}  \) then there is exactly one edge
closed for each vertex that is partial with respect to \( \{u,v\} \); in other words, if 
\( e_{i+1} = \{u,v\}  \) then \( Q(i+1) = Q(i) - 1 - |Y_{u,v}(i)| \).  Therefore, we should have
\[ q(t+\epsilon) n^2 \approx Q \left( i + \epsilon n^{3/2} \right) \approx Q(i) - \epsilon n^{3/2} \cdot 
y(t) n^{1/2} \approx \left( q(t) - \epsilon y(t) \right) n^{2}. \]
This suggests \( dq/dt = - y \).
Now consider the variable \( \left| X_{u,v}(i) \right| \).  Consider a fixed vertex \(w\) that is open with respect to
\( \{u,v\} \).  Note that the probability that the edge \( e_{i+1} \) closes \( \{u,w\} \) (i.e. the
probability of the event \( \{u,w\} \in C_{i+1} \)) is \( | Y_{u,w} |/ |O_i| \).  As the probability that
\( e_{i+1} \in \{ \{u,w\}, \{v,w\} \} \) is comparatively negligible, we suspect that we have
\begin{equation*}
\begin{split}
x(t + \epsilon) n & \approx  \left| X_{u,v} \left( i + \epsilon n^{3/2} \right) \right| \\
& \approx 
\left| X_{u,v}(i) \right| - \epsilon n^{3/2} \cdot
x(t) n \frac{ 2 y(t) \sqrt{n} }{ q(t) n^2} \\
&  \approx  \left(x(t) - \epsilon \frac{2 x(t) y(t)}{ q(t) } \right) n.
\end{split}
\end{equation*}  
This suggests \( dx/dt = - 2xy/q \).
Finally, we consider \( |Y_{u,v}(i)| \).  First note that
a vertex that is partial with respect to \(\{u,v\}\) has its one open edge closed by \( e_{i+1} \)
with probability nearly \( y(t) \sqrt{n} / ( q(t) n^2 ) \).  The probability that a vertex that
is open with respect to \( \{u,v\} \) becomes partial with respect to \( \{u,v\} \) is 
\( 2/ Q(i) \).  So, we should have
\begin{equation*}
\begin{split}
y( t + \epsilon) \sqrt{n} & \approx \left| Y_{u,v}  \left( i + \epsilon n^{3/2} \right) \right| \\
& \approx
\left| Y_{u,v}(i) \right| 
- \epsilon n^{3/2} \cdot y(t) \sqrt{n} \cdot \frac{ y(t) \sqrt{n}}{ q(t) n^2} + \epsilon n^{3/2} \cdot \frac{ 2 x(t) n}{ q(t) n^2} \\
& \approx
\left( y(t) - \epsilon \frac{ y^2(t)}{q(t)} + \epsilon \frac{ 2x(t)}{q(t)} \right) \sqrt{n},
\end{split}
\end{equation*}
which suggests \( dy/dt = - y^2/q + 2x/q \).
As \( |O_0| = n(n-1)/2 \), \( |X_{u,v}(0)| = n-2 \) for all pairs \( \{u,v\} \) and \( |Y_{u,v}(0)| = 0 \) for all
pairs \(\{u,v\}\), our expected value computations suggest that our random 
variables should follow the trajectory given by
\begin{align}
\label{eq:sys1}
\frac{dq}{dt} & = - y  
& \frac{ dx}{dt} & = - \frac{ 2xy}{q}   & \frac{dy}{dt} & = - \frac{ y^2}{q} + \frac{2x}{q} 
\end{align}
with initial conditions \( q(0) = 1/2  \), \( x(0)=1 \) and \( y(0) =0 \).
The solution to this autonomous system is
\begin{equation}
\label{eq:trajectory}
q(t) = \frac{e^{-4t^2}}{2} \ \ \ \ \ \ \ \ \ \   x(t) = e^{-8t^2} \ \ \ \ \ \ \ \ \ \  y(t) = 4te^{-4t^2}. 
\end{equation}
Note that if \( |O_i| \) indeed follows \( q(t) n^2 \) then the triangle-free process will come to end
at \( t = \Theta(\sqrt{ \log n}) \) ; that is, the process will end with \( \Theta( \sqrt{ \log n} \cdot n^{3/2} ) \) 
edges.  (It was Peter Keevash who pointed out that (\ref{eq:sys1}) has this
tantalizing solution  \cite{pete}.)  Observe that the functions \( q(t), x(t), y(t) \) are the appropriate
values for \(G\) chosen uniformly at random from the collection of graphs with \(n\) vertices and 
\( t n^{3/2} \) edges.

We introduce absolute constants \(\a, \b, \g \) and \( \r \).  The constants \( \a \) and \( \r \) are small,
\( \b \) is a large relative to \( \a \) and \( \g \) is large relative to both \( \a \) and \( \b \).  
(These constant can take values \( \a = \r = 1/32 \), \( \b = 1/2 \) and \( \g = 161 \).  No 
effort is made to optimize the constants, and we do not introduce the actual values in an attempt to 
make the paper easier to read).
Set \[ m = \a \sqrt{\log n} \cdot n^{3/2} . \]
Our first result is that our random variables indeed follow the
trajectory (\ref{eq:trajectory}) up to
\( m \) random edges.  In order to state this concentration result we introduce
error functions that slowly deteriorate as the process evolves (in the language of Wormald \cite{w} 
we employ `the wholistic approach' to the differential equations method).  Define
\begin{equation}
\label{eq:errors}
f_q(t) = \begin{cases} e^{41t^2 + 40t} & \text{ if } t \le 1 \\
\frac{e^{41t^2 + 40t}}{t}  & \text{ if } t > 1
\end{cases}
\ \ \ \ \ \   f_x(t) = e^{ 37t^2 + 40t }  \ \ \  \ \ \
f_y(t) = e^{ 41t^2 + 40 t}, 
\end{equation}
and set
\[ g_q(t) = f_q(t) n^{-1/6}  \ \ \ \ \ \ \ \ \  g_x(t) = f_x(t) n^{-1/6}  \ \ \ \ \ \ \ \ \  g_y(t) = f_y(t) n^{-1/6}.\]
Let \( {\mathcal B}_j \) be the event that there exists there exists a step \( i \le j \) 
such that
\[  \left| Q(i) - q(t) n^2 \right| \ge   g_q(t) n^2  \]
or there exists some pair \( \{u,v \} \in \binom{[n]}{2} \setminus E_i \) such that 
\[ \left| \left| X_{u,v}(i) \right| - x(t) n \right|  \ge g_x(t) n \ \ \text{ or } \ \
\left| \left| Y_{u,v}(i) \right| - y(t) \sqrt{n} \right| \ge g_y(t) \sqrt{n} \ \
 \text{ or }  \ \
\left| Z_{u,v}(i)\right| \ge \log^2 n.
\]
\begin{theorem}
\label{thm:lower} If \(n\) is sufficiently large then
\[ Pr\left( {\mathcal B}_{ \a \sqrt{ \log n} \cdot n^{3/2} } \right) \le e^{ -\log^2 n}.   \] 
\end{theorem}
\noindent  
Note that Theorem~\ref{thm:lower} alone places no upper bound on the
number \(M\) of edges in the graph produced by the triangle-free process.  In order to 
achieve such a bound, we bound the independence number of \( G_{m} \).
\begin{theorem}
\label{thm:upper} If \(n\) is sufficiently large then
\[ Pr \left( \alpha \left( G_{ \a \sqrt{ \log n} \cdot n^{3/2} } \right) > \g \sqrt{n \log n} \mid 
\overline{{\mathcal B}_m} \right) < e^{- n^{1/5} }. \]
\end{theorem}
\noindent
Since the neighborhood of each vertex in the triangle-free process is an independent set, it 
follows immediately from Theorem~\ref{thm:upper} that the maximum degree in \( G_M \) is
at most \( \g \sqrt{n \log n} \) a.a.s.  Thus, we have proved Theorem~\ref{thm:show1}.  

The remainder of the paper is organized as follows.  In the next section we establish some technical
preliminaries.  Theorems~\ref{thm:lower}~and~\ref{thm:upper} are
then proved in Sections~\ref{sec:lower}~and~\ref{sec:upper}, respectively.  The proof of Theorem~\ref{thm:extra}
is given in Section~\ref{sec:extra}.

\section{Preliminaries}

Our probability space is the space defined naturally by the triangle-free process.  Let \( \Omega = \Omega_n \) be the 
set of all maximal sequences in \( \binom{[n]}{2}^* \) with distinct entries and the property that
each initial sequence gives a triangle-free 
graph on vertex set \([n] \).  We stress that our measure is not uniform: it is the measure given by 
the uniform random
choice at each step.  We always work with the natural filtration \( {\mathcal F}_0 
\subseteq {\mathcal F}_1 \subseteq \dotsm \)  
given by the process.  Two elements \(x,y\) of \( \Omega \) are in the same part of the partition 
that generates \( {\mathcal F}_j \) iff the first \(j\) entries of \(x\) and \(y\) agree.  We use the
symbol \( \omega_j\) to denote one of the parts in this partition (i.e. \(\omega_j\) denotes a particular
history of the process through \(j\) steps); in particular, if \( \omega \in \Omega \) then \( \omega_j \) is the
part of the partition that defines \( {\mathcal F}_j \) that contains \( \omega \).

For the purpose of notational convenience we use the symbol `\(\pm\)' in two ways: in interval
arithmetic and to define pairs of random variables.  The distinction between the two should be clear from
context.  The degree of a vertex \(v\) in \( G_i \) is denoted \( d_i(v) \) and the neighborhood of \(v\)
in \( G_i\) is \(N_i(v) \).

Our main tool for establishing concentration 
is the following version of the Azuma-Hoeffding inequality.
Let \( \eta, N > 0 \) be constants.  We say that a sequence of random variables
\( A_0, A_1, \dots \) is {\bf \( (\eta, N) \)-bounded} if
\[ A_i - \eta \le A_{i+1} \le A_i + N \ \  \text{ for all } i. \]
\begin{lemma}
\label{lem:sub}
Suppose \( \eta \le N /2 \) and \( a < \eta m \).
If \( 0 \equiv A_0, A_1, \dots \) is an \( (\eta,N) \)-bounded submartingale then
\[ Pr[ A_m \le - a ] \le e^{ - \frac{ a^2}{ 3 \eta m  N}}. \]
\end{lemma}
\begin{lemma}
\label{lem:super}
Suppose \( \eta \le N /10 \) and \( a < m \eta\).
If \( 0 \equiv A_0, A_1, \dots \) is an \( (\eta,N) \)-bounded supermartingale then
\[ Pr[ A_m \ge a ] \le e^{ - \frac{ a^2}{ 3 \eta m  N}}. \]
\end{lemma}
\noindent
As the author
failed to find a reference for these particular inequalities in the 
literature, proofs are
given at the end of the paper, in Section~\ref{sec:ah}.  We often work with pairs 
\( A_0^\pm, A_1^\pm, \dots \) where \(A_0^+, A_1^+, \dots \) is an \( (\eta, N) \)-bounded
submartingale and  \(A_0^-, A_1^-, \dots \) is an \( (\eta, N) \)-bounded 
supermartingale.  We will refer to such a pair of sequences of random variables as
an {\bf \((\eta,N)\)-bounded martingale pair}.

\section{Trajectory}
\label{sec:lower}

%
%

Here we prove Theorem~\ref{thm:lower}, which establishes tight concentration of the random variables \( |O_i| \), 
\( |X_{u,v}(i)| \) and \( |Y_{u,v}(i)| \) around the trajectory given in (\ref{eq:trajectory}) and 
bounds \( | Z_{u,v}(i) | \).

Recall
\( t = t(i) = i/ n^{3/2} \) and \( m = \mu \sqrt{\log n} \cdot n^{3/2} \) and 
\begin{gather*}
g_q(t) = \begin{cases} e^{41t^2 + 40t} n^{-1/6} & \text{ if } t \le 1 \\
\frac{e^{41t^2 + 40t}}{t}  n^{-1/6}  & \text{ if } t > 1
\end{cases}
\ \ \ \ \ \   g_x(t) = e^{ 37t^2 + 40t } n^{-1/6} \ \ \  \ \ \
g_y(t) = e^{ 41t^2 + 40 t} n^{-1/6} . 
\end{gather*}
Note that 
\begin{equation}
\label{eq:relates}
g_q \le \frac{ g_y}{t} \ \ \ \ \ \ \ \text{and } \ \ \ \ \ \ \ g_x = e^{-4t^2} g_y.
\end{equation}

We define events 
\( \cx \), \( \cy \) and \({\mathcal Z} \).  For \( \omega \in {\mathcal B}_m \) let 
\(\ell\) be the smallest index such that \( \omega \in {\mathcal B}_\ell \) but \( \omega 
\not\in {\mathcal B}_{\ell-1} \); in other words, the random variable 
\( \ell \) is the first time that one of our tracked random variables
is outside the allowable range.  We define \( {\mathcal X} \) to be the set of \( \omega \in {\mathcal B}_m \)
such that there exists a pair \( \{u,v\} \) such that \( \{u,v\} \not\in E_\ell \) and
\[ \left| X_{u,v}(\ell) \right| \not\in n \left[ x( t(\ell)) \pm g_x( t(\ell)) \right]. \]
So, an atom \( \omega \in \cb_m \) is in \( {\mathcal X} \) if there is some pair of vertices 
\( \{u,v\} \)
such that the number of open 
vertices with respect to \( \{u,v\} \) is a reason we place \( \omega \in {\mathcal B}_\ell \).  
Define \( {\mathcal Y} \) and \( {\mathcal Z} \) analogously.   We prove Theorem~\ref{thm:lower} by showing
\begin{equation}
\label{eq:noq}
 {\mathcal B}_m = \cx \cup \cy \cup \cz 
\end{equation}
and then bounding the probabilities of \( \cx\), \( \cy\) and \( \cz\).  
In the next subsection we show that if \( \left|Y_{u,v}(j)\right| \) is in range for all \( j \le i \) and all pairs
\( \{u,v\} \) then \( |O_i| \) is in range, thereby establishing (\ref{eq:noq}).
In the following three 
subsections we establish upper bounds on the probabilities of the events \( \cx, \cy\) and  \(\cz \),  
respectively.

\subsection{Open edges}

Here we simply take advantage of the strict control we enforce on the number of
partial vertices at each pair; we do not invoke any concentration
inequalities in this subsection.  Note that if we have \( e_{i+1} = \{u,v \} \) then the number of edges closed when we
add \( e_{i+1}\)
is simply equal to \( |Y_{u,v}(i)| \), the number of partial vertices at 
\( \{u,v\} \).  Therefore, assuming \( \omega \not\in \cb_{j-1} \), we 
have
\begin{equation*}
\begin{split}
| O_j | & = \frac{n(n-1)}{2} - j - \sum_{i=0}^{j-1} \left| Y_{ e_{i+1}} (i) \right| \\
& \in \frac{ n^2}{2} - \frac{n}{2} - j - \sqrt{n} \left[ \sum_{i=0}^{j-1} y( t ) \pm g_y( t ) \right] \\
& \subseteq \frac{ n^2}{2} - n^2 \int_0^{ t(j) } 4\t e^{-4\t^2} d\tau \pm n^{11/6} \int_0^{t(j)} e^{41 \t^2 + 40\t} d\tau \pm n^{5/3} \\
& \subseteq n^2 \left[ q(t(j)) \pm  g_q(t(j)) \right]. 
\end{split}
\end{equation*}
Note that this establishes (\ref{eq:noq}).

\subsection{Open vertices}

Consider a fixed \( \{u,v\} \in \binom{[n]}{2} \).  We write 
\[ \left| X_{u,v}(j) \right| = n -2 - \sum_{i=1}^j A_i \]
where \( A_i \) is the number of open vertices at \( \{u,v\} \) that 
are eliminated when \( e_i \) is added to the process.  Define \( A_i^+ \) and \( A_i^- \) by
\begin{gather*}
A^{\pm}_i = 
\begin{cases}
A_i +  \frac{1}{ \sqrt{n}} \left[ - \frac{ 2 x(t) y(t)} {q (t)} \pm (17 t + 39) g_x(t)  \right] & \text{ if } \omega 
\not\in {\mathcal B}_{i-1} \text{ and } \{u,v\} \not\in E_{i} \\
0 & \text{ if } \omega \in {\mathcal B}_{i-1} \text{ or } \{u,v\} \in E_{i}
\end{cases} \\
B^{\pm}_j = \sum_{i=1}^j A^{\pm}_i .
\end{gather*}
Note that if \( \omega \not\in \cb_{j-1} \) and \( \{u,v\} \not\in E_j \) then we have
\begin{equation*}
\begin{split}
\left| X_{u,v}(j) \right| & = n -2 - \sum_{i=1}^j A_i^+ - \sum_{i=1}^j 
\left[ \frac{ 2 x\left( t \right) y( t )} {q (t)} - (17t + 39) g_x(t) \right] \frac{1}{ \sqrt{n}} \\
& \le n - B_j^+  -  n \int_{0}^{t(j)}  \frac{ 2 x( \tau) y( \tau)}{ q( \tau)} d \tau + n^{5/6} \int_{0}^{t(j)} (17\tau+39) 
e^{37 \tau^2 + 40\tau} d \tau \\
& \le n x( t(j)) + n^{5/6} \left(e^{37 t^2(j) + 40 t(j)} - 1 \right) - B_j^+ \\
& = n \left[ x(t(j)) +  g_x( t(j)) \right]  - \left( B_j^+ + n^{5/6} \right).
\end{split}
\end{equation*}
Therefore, the event \( |X_{u,v}(j)| > n \left[ x(t(j)) + g_x( t(j)) \right] \) is contained in the event \( B_j^+ < - n^{5/6} \).  
Similarly, the event \( |X_{u,v}(j)| < n \left[ x(t(j)) - g_x(t(j)) \right] \) is contained in the event \( B_j^- > n^{5/6} \).
We bound the probabilities of these events by application of the martingale inequalities.
\begin{claim}
\( B^\pm_0, B^\pm_1, \dots \) is a \( ( \frac{4}{\sqrt{n}}, \sqrt{n}) \)-bounded martingale pair.  
\end{claim}
\begin{proof}
We begin with the martingale condition.  Of course, we can restrict our attention 
to \( \omega_i \) such that \( \omega_i \not\subseteq \cb_{i}\) and \( \{u,v\} \not\in E_i \).
Consider a vertex \( w \in X_{u,v}(i) \).  Note that
\( w \not\in X_{u,v}(i+1)\) if 
\( e_{i+1} \in \left\{ \{u,w\}, \{v,w\} \right\} \), \( e_{i+1} \) connects
\( \{u,w\} \) to one of the vertices that is partial at 
\( \{u,w\} \) or \( e_{i+1} \) 
connects \( \{v,w\} \) to one of the vertices that is partial at \( \{v,w\} \).  Note that (as we assume 
\( \{ u,v \} \not\in E_i \)) the edge \( e_{i+1} \) plays 2 of these roles if and only if
\( e_{i+1} = \{z,w\} \) where \( z \in Z_{u,v}(i) \).  It follows that we have
\[ Pr \left( w \not\in X_{u,v}(i+1) \right) = \frac{2 + | Y_{u,w}(i)| + | Y_{v,w}(i)| - |Z_{u,v}(i)|}{| O_i|}, \]
and therefore
\[ E[ A_{i+1} \mid \fee_i ] = \frac{1}{|O_i|} \left[ \sum_{w \in X_{u,v}(i) }  2 + | Y_{u,w}(i)| + | Y_{v,w}(i)| - |Z_{u,v}(i)| \right]. \]
As we restrict our attention to \( \omega_i \not\subseteq \cb_i \), we have
\begin{equation*}
\begin{split}
E[ A_{i+1} \mid \fee_i ] & \in  
\frac{ 2 n^{3/2} ( x \pm g_x )( y \pm g_y ) }{  n^2 ( q \pm g_q ) }  
+ \left( -\frac{ n (x+g_x) \log^2 n }{ n^2 (q - g_q)}, \frac{ 2 n (x + g_x) }{ n^2 (q - g_q)} \right) 
\\
& \subseteq \frac{1}{\sqrt{n}} \left[ \frac{2 x y }{q } \pm \left(  \frac{ 2g_yx  
+ 2g_xy  + 2g_xg_y }{ q - g_q } + \frac{ 2xy g_q }{ q( q - g_q )} \right) \right] + \left( - \frac{ \log^2 n}{n} \left[ \frac{ 4 x}{q} \right], 
\frac{ 1}{n} \frac{ 4x}{q} \right) \\
& \subseteq \frac{1}{\sqrt{n}} \left[ \frac{2 x y }{q }  \pm \left( 
5 e^{-4t^2} g_y  +  17t g_x  + 5 g_xg_y e^{4t^2} + 33 t e^{-4t^2} g_q \right) \right] 
\pm \frac{ \log^2n }{n} 4 e^{-4t^2} \\
& \subseteq \frac{1}{\sqrt{n}} \left[ \frac{2 x y }{q }  \pm \left( 17t g_x + 39 g_x \right) \right].
\end{split}
\end{equation*}
(Note that we apply (\ref{eq:relates}).)  This establishes the martingale condition.

Now we turn to the bounds on \( A^\pm_{i+1} \).  We use the simple fact that 
the set of edges closed when we add \( e_{i+1} \) is determined by \( Y_{e_{i+1}}(i) \); one edge in each partial triangle in
\( Y_{e_{i+1}}(i) \) is closed.  Therefore, the maximum value of \(A_i\) is bounded above by 
\( (y(t) + g_y(t))\sqrt{n} \), which is at most \( \sqrt{n} \).  Of course \( A_i^\pm \) takes its smallest
value when \( A_i = 0 \), and in this case we have \( A^\pm_i > - 4 / \sqrt{n} \) as \( 2 xy/q \le 4 /\sqrt{e} \).
\end{proof}

\noindent
Applying Lemmas~\ref{lem:sub}~and~~\ref{lem:super} we have
\begin{equation}
\label{eq:steps}
Pr( B^+_m < - n^{5/6}), Pr( B^-_m  > n^{5/6} ) \le e^{ - \frac{n^{5/3}}{12m}  }. 
\end{equation}

We claim that \( \cx \) is contained in the union, taken over all pairs \( \{u,v\} \), of the 
events given in (\ref{eq:steps}).  Indeed, if \( \omega \in \cx \) on account of the the pair
\( \{u,v\} \) at step \( \ell \) then either \( B^+_\ell < - n^{5/6} \) or \( B^-_\ell > n^{5/6} \) and
we also have \( B_j^+ = B_{\ell}^+ \) and \( B_j^- = B_\ell^- \) for all \( j \ge \ell \) 
(as we set \( A_{j+1}^{\pm} = 0 \) in the event \( \cb_j\)). Therefore, we have
\[ Pr( \cx) \le 2 \binom{n}{2} e^{ - \frac{ n^{5/3}}{ 12 m} }. \]

\subsection{Partial vertices}

We use the same reasoning as in the last subsection, but here we break the step by 
step changes in 
\( |Y_{u,v}(i)| \) into two parts.
We write \( \left| Y_{u,v}(j) \right| \) as a sum
\[ \left| Y_{u,v}(j) \right| = \sum_{i=1}^j U_i - V_i , \]
where \( U_i \) is the number of partial vertices at \( \{u,v\} \) created when 
\( e_i \) is added and \( V_i \) is the number of 
partial vertices at \( \{u,v \} \) 
eliminated when \( e_i \) is added.  Note that if \( \{u,v\} \in E_i \) then we set
\( U_i = V_i = 0 \) (in order to maintain consistency with the definition of \( Y_{u,v} \)).

We begin with an analysis of \( V_i\).
Define \( W_0^\pm = 0 \) and
\begin{gather*}
V_i^\pm = \begin{cases} V_i + \left[ - \frac{ y^2(t) }{ q( t) } \pm (82t + 1) g_y(t) \right] \frac{1}{n} 
& \text{ if } \omega 
\not\in {\mathcal B}_{i-1} \text{ and } \{u,v\} \not\in E_{i} \\
0 & \text{ if } \omega \in {\mathcal B}_{i-1} \text{ or } \{u,v\} \in E_{i}
\end{cases} \\
W_j^\pm = \sum_{i=1}^j V_i^\pm.
\end{gather*}

\begin{claim}
\( W^\pm_0, W^\pm_1, \dots \) is a \( ( \frac{4}{ n }, \log^2 n) \)-bounded martingale pair.  
\end{claim}
\begin{proof}
We begin with the martingale conditions.  Suppose \(w\) is partial with respect to \( \{u,v\} \).  Let
\( w^* \)  be the unique 
vertex in \( \{u,v \} \) such that \( \{w^*,w\} \in O_i \).
Note that \(w\) is removed from \( X_{u,v} \) if either
\( e_{i+1}= \{w,w^*\} \) or \(e_{i+1}\) is one of the pairs in \( O_i \) that 
links \( \{w^*,w \} \) to \( Y_{w^*,w}(i)\) (other than \( \{ u,v\} \) itself).
Therefore, restricting our attention to \( \omega_i \not\subseteq \cb_i \), we have
\[ E\left[ V_{i+1} \mid \fee_i \right] = \sum_{ w \in Y_{u,v}} \frac{  |Y_{w^*,w}| }{ |O_i| }. \]
As we restrict our attention to \( \omega_i \not\subseteq \cb_i \) and \( \{u,v\} \not\in E_i \) we have
\begin{equation*}
\begin{split}
E\left[V_{i+1} \mid \fee_i \right] & \in  \frac{ \sqrt{n} (y \pm g_y) \left( \sqrt{n}(y \pm g_y) \right)}{ n^2 (q \pm g_q)} \\
& \subseteq \frac{1}{n} \left[ \frac{y^2}{q} \pm  \left( \frac{ 2g_yy  
+ g_y^2 }{ q - g_q } + \frac{ y^2 g_q }{ q( q - g_q )} \right) \right]
\\
& \subseteq  \frac{1}{n} \left[ \frac{y^2}{q} \pm \left(  17 t g_y + g_y + 65 t g_y \right) \right], 
\end{split}
\end{equation*}
and the martingale conditions are established.

It remains to establish boundedness.  Note that if \( e_{i+1} \) does not intersect \( \{u,v\} \) then the
change in \( |Y_{u,v} (i) | \) is at most 2 (as all edges that are closed when we add \( e_{i+1} \) intersect
\( e_{i+1} \)).  So, suppose \( e_{i+1} = \{u,z\} \) where \( z \neq v \).  
If the vertex \(w\) is then removed from \( Y_{u,v} \) then the edge \( \{w,u\} \) must have been
closed by \( e_{i+1} \).  This implies \( \{w,z\} \in E_i \).  Furthermore, as \( w \) is partial with respect to
\( \{u,v\} \), we have \( \{ w,v \} \in E_i \).  Thus 
\(w \in Z_{z,v}(i) \).  Therefore, the change in \( V_i \) is bounded by the maximum value of 
\( \left| Z_{x,y}(i) \right| \),
which is bounded by \( \log^2 n\).  The lower bound follows from \( y^2/q \le 8/e \).
\end{proof}
\noindent
Applying Lemmas~\ref{lem:sub}~and~~\ref{lem:super} we have
\[ Pr( W^+_m < -n^{1/3}/2), Pr( W^-_m  > n^{1/3}/2 ) \le e^{ - \frac{n^{2/3}}{48 m \log^2 n/ n}  }.  \]

Now we turn to \( U_i \).  
Define \( T_0^\pm = 0 \) and
\begin{gather*}
U_i^\pm = \begin{cases} U_i + \left[ - \frac{ 2 x(t ) }{ q( t) } \pm 14 g_y(t) \right] \frac{1}{n} 
& \text{ if } \omega 
\not\in {\mathcal B}_{i-1} \text{ and } \{u,v\} \not\in E_{i} \\
0 & \text{ if } \omega \in {\mathcal B}_{i-1} \text{ or } \{u,v\} \in E_{i}
\end{cases} \\
T_j^\pm = \sum_{i=1}^j U_i^\pm.
\end{gather*}

\begin{claim}
\( T^\pm_0, T^\pm_1, \dots \) is a \( ( \frac{5}{ n }, 1 ) \)-bounded martingale pair.  
\end{claim}
\begin{proof}
We begin with the martingale conditions.  As usual we restrict our attention to \( \omega_i \not\subseteq \cb_i \).
We have
\[ E\left[ U_{i+1} \mid \fee_i \right] = \frac{  2 |X_{u,v}(i)| }{ |O_i| }, \]
and
\begin{equation*}
\begin{split}
E\left[U_{i+1} \mid \fee_i \right] & \in \frac{ 2 n (x \pm g_x) }{ n^2 (q \pm g_q)} \\
& \subseteq \frac{1}{n} \left[ \frac{2x}{q} \pm \frac{ 2g_q x + 2q g_x }{ q(q - g_q)} \right] \\
& \subseteq \frac{1}{n} \left[ \frac{ 2x }{q}  \pm \left( 9 g_q + 5 e^{4t^2} g_x \right) \right] \\
& \subseteq \frac{1}{n} \left[ \frac{ 2x }{q} \pm 14 g_y \right],
\end{split}
\end{equation*}
which establishes the martingale conditions.  

As the addition of \( e_{i+1} \) to the graph can create at 
most one new partial vertex at \( \{u,v\} \), \( U_i \) is either 1 or 0.  
Furthermore,
\( 2 x/q = 4 e^{-4t^2} \le 4 \).  These two observations establish 
the boundedness condition.
\end{proof}
\noindent
Applying Lemmas~\ref{lem:sub}~and~~\ref{lem:super} we have
\[ Pr( T^+_m < -n^{1/3}/2), Pr( T^-_m  > n^{1/3}/2 ) \le e^{ - \frac{n^{2/3}}{60 m/n }  }  \]

Now we are ready to return to the random variable \( \left| Y_{u,v}(i) \right| \) itself.  We have
\begin{equation*}
\begin{split}
\left| Y_{u,v}(j) \right| & = \sum_{i=1}^j U_i - V_i \\
& = \sum_{i=1}^j U_i^+ + \frac{1}{n} \left[ \frac{ 2x}{q} - 14 g_y \right]
- \left( \sum_{i=1}^j V_i^- + \frac{1}{n} \left[ \frac{ y^2}{q} + ( 82t +1) g_y \right]  \right) \\
& = T_j^+ - W_j^- + \frac{1}{n} \sum_{i=1}^j \left( \frac{ 2x}{q} - \frac{ y^2}{q} \right) - \frac{1}{n} \sum_{i=1}^j ( 82t + 15) g_y \\
& \ge \sqrt{n} \int_0^{t(j)}  \frac{ 2x}{q} - \frac{y^2}{q}  d\t  - n^{1/3} \int_0^{t(j)} (82\t + 15) e^{41\t^2 + 40\t} d\t
+ \left( T_j^+ - W_j^- \right) \\
& \ge \sqrt{n} \left[ y( t(j)) - g_y( t(j)) \right] + n^{1/3} + T_j^+ - W_j^-
\end{split}
\end{equation*}
Therefore, the event \( |Y_{u,v}(j)| < \sqrt{n} \left[ y(t(j)) - g_y( t(j)) \right] \) is contained in
\[ \left\{ T_j^+ < - n^{1/3}/2 \right\} \vee \left\{ W_j^- > n^{1/3}/2 \right\}. \]
We have already bounded the probabilities of these events.  The analogous argument holds for the event
\( |Y_{u,v}(j)| > \sqrt{n}( y(t(j)) + g_y( t(j)) ) \), with \( T_j^+ \) and \( W_j^- \) replaced with \( T_j^- \) 
and \( W_j^+ \), respectively.  As the random variables \( T^\pm_i \) and \( W^\pm_i \) are `frozen' 
once one of the random variables leaves the allowable range, we have
\[ Pr\left( \cy \right) \le \binom{n}{2} \cdot 2 \left(  e^{ - \frac{n^{2/3}}{48 m \log^2 n/ n}  } +  
e^{ - \frac{n^{2/3}}{60 m/n }  } \right) < 2 n^2 e^{ - n^{1/6}/48 }. \]

\subsection{Complete vertices}

Note that the probability\ that \( e_{i+1} \) adds a complete vertex at \( \{u,v \} \) is at most
\( |Y_{u,v}(i)| / | O_i | \).
So, in the event \( \cb_i\), we have
\[ Pr\left( \left| Z_{u,v}(i+1) \right| = \left|Z_{u,v}(i)\right| +1 \right) 
\le \frac{ \sqrt{n} ( y(t) + g_y(t)) }{ n^2 ( q(t) - g_q(t)) } \le \frac{ 9t}{ n^{3/2}}. \] 
Therefore, 
\[ Pr \left[ \left| Z_{u,v}(m) \right| \ge \log^2 n \right] 
\le \binom{ \a n^{3/2} \sqrt{ \log n}}{ \log^2 n} 
\left( \frac{ 9 \a \sqrt{\log n}}{n^{3/2} }\right)^{ \log^2 n } \le e^{ - \frac{1}{2} (\log^2 n) \log \log n} \]
for \(n\) sufficiently large.  Thus
\[ Pr( \cz ) \le \binom{n}{2}  e^{ - \frac{1}{2} (\log^2 n) \log \log n}. \]

\section{Independent Sets}
\label{sec:upper}

Our goal is now to prove Theorem~\ref{thm:upper}.  We will bound
from above the probability, conditional on $\overline{B_m}$, that
any fixed set $K$ of $\gamma\sqrt{n\log n}$ vertices is independent.
This bound will be so small that it remains small when multiplied
by the number of such $K$.  The conditioning on $\overline{B_m}$
tells us that the variables $Q, X_{u,v},Y_{u,v}$ all remain
quite close to $q(t)n^2,x(t)n,y(t)\sqrt{n}$ throughout the
process.  As it happens, the strength of the error terms $g_q,g_y,g_x$ 
does not play a major role in the calculations
below.  The reader might, at first reading, set $g_q=g_y=g_x=0$
so as to get a less cluttered view of the techniques involved.

Recall that \(\a, \b, \g\) and 
\(\r\) are constants where \( \a \) and \( \r \) are small, \( \b \) is large relative to
\( \a \) and \( \g \) is large relative to \( \a \) and \( \b \).   Also recall
\( m = \a \sqrt{ \log n} \cdot n^{3/2} \).  We make 2 initial observations (Claims~\ref{cl:degrees}~and~\ref{cl:obsv2}).
Let \( \cd_i \) be the event that \( G_i \) has a vertex 
of degree greater than \( \beta \sqrt{ n \log n } \).
\begin{claim} If \(n\) is sufficiently large then
\label{cl:degrees}
\[ Pr( \cd_m \wedge \overline{\cb_m} ) \le e^{ -n^{1/5}}  . \]
\end{claim}
\begin{proof}
We begin by establishing an upper bound on the 
number of open pairs at each vertex.  
For each vertex \(v\) let 
\( W_v(i) \) be the set of pairs in \( O_i \) that contain \(v\).
Let \( A_i \) be the number of open pairs that contain \(v\) that are removed from \( W_v \)
when the edge \( e_i \) is added to the process.  Note that we have
\[ E \left[ A_{i+1} \mid \fee_i \right] = \sum_{w \in W_v(i)} \frac{ 1 + |Y_{v,w}(i)| }{ |O_i| } \]
Define
\[  B_{i+1} = \begin{cases}
A_{i+1} - \frac{1}{\sqrt{n}} \left(  8t e^{-4t^2} - 20 g_y \right)
& \text{ if }  W_v(i) >  e^{-4 t^2} n \text{ and } \omega \not\in \cb_{i} \\
0 & \text{ if }  W_v(i) \le  e^{-4 t^2} n \text{ or } \omega \in \cb_{i}
\end{cases} \]
Note that (restricting our attention to \( \omega_i \not\subseteq \cb_i \) and \(  W_v(i) >  e^{-4 t^2} n\))
\begin{equation*}
\begin{split}
E\left[ B_{i+1} \mid \fee_{i} \right] & \ge  e^{- 4 t^2} n \left( \frac{ 1 + \sqrt{n}( y - g_y )}{ n^2( q + g_q)} \right)
- \frac{1}{\sqrt{n}} \left(  8t e^{-4t^2} - 20 g_y \right) \\
& \ge \frac{e^{-4t^2}}{ \sqrt{n}} \left( \frac{y}{q} - \frac{ g_y}{ q + g_q} - \frac{ g_q y}{ q ( q + g_q) } \right) 
- \frac{1}{\sqrt{n}} \left(  8t e^{-4t^2} - 20 g_y \right) \\
& \ge \frac{ 1}{ \sqrt{n}} \left( - 3 g_y - 17 t g_q + 20 g_y \right) \\
& \ge 0.
\end{split}
\end{equation*}
Therefore, any sequence of the form \( B_\ell, B_\ell+ B_{\ell+1}, \dots, \sum_{i=\ell}^j B_i, \dots \) 
is a \( ( 2/ \sqrt{n}, \sqrt{n} ) \)-bounded
submartingale.  Therefore, for any \( \ell < j \) we have
\[ Pr\left( \sum_{i=\ell}^j B_i  \le - n^{7/8} \right) \le e^{ - \frac{n^{7/4}}{ 6m }}. \]

Now consider the event 
\(  |W_v(j)| > e^{- 4t(j)^2} n + 2 n^{7/8} \).  In this event there exists a maximum
\(\ell < j \) such that \( |W_v(\ell) | \le e^{-4t(\ell)^2} n \).  We have
\[ \sum_{ i = \ell+2}^j A_i < \left( e^{-4 t(\ell)^2} - e^{ - 4 t(j)^2} \right) n  - 2 n^{7/8}, \]
which implies
\begin{equation*}
\begin{split}
\sum_{i = \ell+2}^j  B_i & < \left( e^{-4 t(\ell)^2} - e^{ - 4 t(j)^2} \right) n  - 2 n^{7/8} 
-  \frac{1}{\sqrt{n}} \sum_{i=\ell+1}^{j-1}  \left(  8t e^{-4t^2} - 20 g_y \right) \\
& < - \frac{3}{2} n^{7/8} + n \int_{ t(\ell+1)}^{ t(j) } 20 g_y(\t) d\t \\
& < - n^{7/8}.
\end{split}
\end{equation*}
Let \( \cd_m^\prime \) be the event that 
there exists a vertex \(v\) and a step \(j \le m\) such that 
\(|W_v(j)| >  e^{- 4t(j)^2} n + 2 n^{7/8} \).  We have shown
\[ Pr \left( \cd_m^\prime \wedge \overline{\cb_m} \right) \le n \binom{m}{2} 
\exp \left\{ - n^{7/4}/ (6 m) \right\} . \] 

So, we can restrict our attention to the event \(\overline{\cd_j^\prime} \).  Note that here we
have \( |W_v(j)| \le 4 |O_j |/n \) for all \( j,v \).  Now we 
simply use the 
union bound.
\begin{multline*}
Pr\left( \cd_m \wedge \overline{\cb_m} \wedge \overline{\cd_m^\prime}  \right) \le n \binom{ \a \sqrt{ \log n } \cdot n^{3/2}}{ \b \sqrt{ n \log n}} 
\left( \frac{ 4}{n} \right)^{\b \sqrt{ \log n \cdot n }} \\ 
\le n \left(  \frac{ \a \sqrt{ \log n} \cdot n^{3/2} \cdot 4e}{ \b \sqrt{n \log n} \cdot 
n} \right)^{\b \sqrt{ n \log n}}  = n\left( \frac{ \mu 4 e }{ \beta} \right)^{\b \sqrt{ n \log n}}. 
\end{multline*}
\end{proof}

Next we consider the number of open pairs in sufficiently large bipartite subgraphs.  
Let \( A, B \) be disjoint subsets of \( [n] \) such that 
\[  |A| = |B| = \left( \frac{ \g - \b}{2} \right) \sqrt{n \log n} = k. \]
We track the evolution of the number of pairs in
\( O_i \) that intersect both \(A\) and \(B\).  Note that a vertex with large degree in either 
\(A\) and \(B\) can cause a large one step change in this variable.  To deal
with this possibility, we introduce the following definition.  
Let \( A \times B \) be the set of pairs \( \{u,v\} \in \binom{[n]}{2} \) that intersect both \(A\) and
\(B\). We say that the
pair \( \{u,v\} \in A \times B \) is {\bf closed with respect to \(A,B\)} if
there exists \(x \not\in A \cup B \) and \( j \le i \) such that 
\[ |N_j(x) \cap A|, | N_j(x) \cap B| \le \frac{k}{n^{\r}} \ \ \ \ \ \ \ \text{ and }
\ \ \ \ \ \ \ u,v \in N_j(x). \]
A pair \( \{u,v\} \in A \times B \) is {\bf open with respect to \( A,B \)} if \( \{u,v \} \not\in E_i\) and 
\( \{u,v\} \) is not closed
with respect to \( A,B\).
Define
\[ W_{A,B}(i) = \left\{ \{u,v\} \in A \times B : \{u,v\} \text{ is open with 
respect to } A,B \text{ in } G_i \right\}. \]
Note that a pair \( \{u,v\} \in A \times B \) can be 
closed (i.e. in \( C_i \)) and still be in 
\( W_{A,B}(i) \).  We stop tracking \( W_{A,B} \) as soon as a single edge falls in \( A \cup B \); formally,
if \( E_i \cap \binom{ A \cup B}{2} \neq \emptyset \) or \( \omega_i \subseteq \cd_i \vee \cb_i \) then we set
\( W_{A,B}(i) = W_{A,B}(i-1) \).

Let \( \cp_j \) be the event there exist \(A,B \in \binom{[n]}{k} \) and a step \(i \le j\) such that  
\[ \binom{ A \cup B}{2} \cap E_i = \emptyset \ \ \ \text{ and } \ \ \ 
\left| W_{A,B}(i) \right| < e^{-4t^2} 
k^2 - 2n^{1 - \r/3} \]
\begin{claim} 
\label{cl:obsv2}
If \(n\) is sufficiently large then
\[ Pr\left( \cp_m \wedge \overline{ \cb_m} \right) \le e^{- n^{1/2}}.  \]
\end{claim}
\begin{proof}

Let \(X_i\) be the number of pairs that leave \( W_{A,B} \) at step \(i\)  of the process.  
We have
\[ E[ X_{i+1} \mid \fee_i ] \le \sum_{ \{u,v\} \in W_{A,B}} \frac{ Y_{u,v}(i) }{ |O_i| }.\]
Note that we only have an upper bound here as there may be edges between \( \{u,v\} \) and  \( Y_{u,v} \) 
that would close
\( \{u,v\} \) without removing \( \{u,v\} \) from \( W_{A,B} \).  Define
\[ Y_{i+1} = \begin{cases} X_{i+1} - \frac{ \log n}{ \sqrt{n}}  \left( \frac{ (\g - \b)^2 }{4} \right) 
\left[ 8t e^{-4t^2} + 20 g_y \right] & \text{ if } |W_{A,B}(i)| < e^{-4t^2}  
k^2  \text{ and } \omega \not\in \cb_i   \\
0 & \text{ if } |W_{A,B}(i)| \ge e^{-4t^2} 
k^2  \text{ or } \omega \in \cb_i 
\end{cases} \] 
Note that
\begin{equation*}
\begin{split}
E[ Y_{i+1} \mid \fee_i ] & \le 
e^{-4t^2} 
k^2 \left( \frac{  \sqrt{n}( y + g_y)}{ n^2 ( q - g_q)} \right)
- \frac{ \log n}{ \sqrt{n}}  \left( \frac{ (\g - \b)^2 }{4} \right) 
\left[ 8t e^{-4t^2} + 20 g_y \right]
\\
& \le  \left( \frac{ (\g - \b)^2 }{4} \right) \frac{ \log n }{ \sqrt{n} }  
\left(  e^{-4t^2} \left[ \frac{y}{q} + \frac{ g_y}{ q - g_q} + \frac{ g_q y}{ q ( q - g_q) } \right]
- \left[ 8t e^{-4t^2} + 20 g_y \right] \right) \\
& \le  \left( \frac{ (\g - \b)^2 }{4} \right) \frac{ \log n }{ \sqrt{n} }  
\left(  3 g_y + 17 t g_q - 20 g_y \right) \\
& \le 0.
\end{split}
\end{equation*}
Therefore, \( Y_\ell, Y_\ell+ Y_{\ell+1}, \dots, \sum_{i = \ell}^j Y_i, \dots \) is a 
\( ( \frac{ 2 \g^2 \log n}{ \sqrt{n}},  k n^{- \rho} ) \)--bounded supermartingale.  It follows 
that we have
\begin{equation}
\label{eq:probs}
Pr \left[ \sum_{i=\ell}^j Y_i > n^{1 - \r/3} \right] \le \exp \left\{ - \frac{ n^{2 - 2\r/3} }{ 3 \g^3 m \log^{3/2} n \cdot n^{-\r} } \right\} 
= \exp \left\{ - \frac{ n^{1/2 + \r/3}}{ 3 \a \g^3 \log ^2 n } \right\}.
\end{equation}

Now we turn to the event \( \cp_j \wedge \overline{ \cb_j} \), where we assume that it is step \(j\) 
where \( |W_{A,B}(j)| \) is too small for the first time.  There exists a maximum \( \ell < j \) such that 
\( | W_{A,B}(\ell) | \ge e^{-4t(\ell)^2} 
k^2 \).  Then
\begin{equation*}
\begin{split}
\left| W_{A,B}(j) \right| & >  e^{-4t(\ell)^2} 
k^2  - \sum_{i=\ell+1}^j X_i \\
& =  e^{-4t(\ell)^2} 
k^2 - X_{\ell+1}  - \sum_{i=\ell+2}^j Y_i - 
\frac{ \log n}{ \sqrt{n}}  \left( \frac{ (\g - \b)^2 }{4} \right)  \sum_{i=\ell+2}^j \left[ 8t e^{-4t^2} + 20 g_y \right] \\
& >  e^{-4t(j)^2} 
k^2 - \sum_{i=\ell+1}^j Y_i - k^2 \int_{t(\ell)}^{t(j)} 20 g_y( \t) d\t - \sqrt{n}.
\end{split}
\end{equation*}
Therefore, applying (\ref{eq:probs}), we have
\[ Pr\left( \cp_m \wedge \overline{\cb_m} \right) \le 
\binom{n}{ k
}^2 \cdot \left( n^{3/2} \sqrt{\log n} \right)^2 \cdot 
\exp \left\{ -  \frac{ n^{1/2 + \r/3}}{ 3 \a \g^3 \log^2 n}  \right\}. \]

\end{proof}

Consider a fixed set \(K\) of \(\g \sqrt{n \log n} \) vertices.  We bound the probability that
\(K\) is independent in \( G_m \) by first showing that if \( K \) is independent in \( G_i \) (and we
are not in the `bad' event \( \cb_i \vee \cd_i \vee \cp_i \))
then the number of pairs in \( \binom{K}{2} \cap O_i \) is at least
a constant time \( e^{-4t^2} |K|^2 \).    This implies that the edge \( e_{i+1} \) 
has a reasonably good chance of
falling in \(K\).  

We restrict our attention to 
\( \overline{\cb_m} \wedge \overline{\cd_m} \wedge \overline{\cp_m} \).
For each step \(i\) of the process such that \( \binom{K}{2} \cap E_i = \emptyset \) 
let \( L_i\) be the set of vertices \(x\) such that
\( x \not\in K \) and \( | N_i(x) \cap K | > k/ n^\r \).  Set
\[ \cn_i = \left\{ N_i(x) \cap K : x \in L_i \right\}. \]
We first note that, since co-degrees are bounded 
when we are not in the event
\( {\mathcal B}_i \), we have
\begin{gather*}
X,Y \in {\mathcal N}_i \ \ \ \Rightarrow \ \ \ |X \cap Y| \le \log^2 n.
\end{gather*}
It follows that the cardinality of the union of \( f \) sets in 
\( {\mathcal N}_i  \) is
at least \( f k/ n^\r - f^2 \log^2 n \), and therefore
\[ \left| L_i \right| \le 2 n^\r. \] 
Furthermore, as we restrict our attention to \( \overline{ \cd_i} \) ,
we have
\[ X \in {\mathcal N}_i \ \ \ \ \Rightarrow \ \ \ \  |X| \le \b \sqrt{ n \log n}. \]
Now, we identify disjoint sets \( A,B \) such that the set of pairs \(A \times B\)
is essentially disjoint from \( \binom{X}{2} \) for all \( X \in {\mathcal N}_i \).
Form \( A \subseteq K \) such that \( |A| = k \) by iteratively adding sets from \( {\mathcal N}_i \) for
as long as possible.  Let \( B \subseteq K \setminus A \) have the property that \( |B|=k\) and \( B \cap X = \emptyset \) for all
\( X \in {\mathcal N}_i \) that are used to form \(A\).  Note that we have
\begin{gather*}
X \in {\mathcal N}_i \ \ \ \ \ \ \Rightarrow \ \ \ \ \ \  | X \cap A| \le \log^2n\cdot 2 n^{\rho} \ \text{ or } \  
| X \cap B| = 0. 
\end{gather*}
Note that the number of edges in \( W_{A,B}(i) \) that are in \( C_i \) is
at most 
\[ | L_i|  \left( 2 \log^2n \cdot n^{\rho} \right) \cdot \beta \sqrt{n \log n} \le 4 \b \log^{5/2}n \cdot n^{1/2 + 2 \rho}. \]
Therefore, since \( \omega_i \not\subseteq \cp_i \),
\begin{equation*}
\begin{split}
\left| O_i \cap \binom{ K}{2} \right| & \ge  e^{-4t^2} \left( \frac{ ( \gamma -\beta)^2 }{4} \right) n \log n -  2n^{1 - \r/3} - 
4 \beta \log^{5/2}n \cdot n^{1/2 + 2 \rho} \\
&  \ge e^{-4t^2} \cdot \frac{ (\gamma-\b)^2}{5} \cdot n \log n. 
\end{split}
\end{equation*}
Thus, since \( \omega_i \not\subseteq \cb_i\),
\[  Pr \left( e_{i+1} \in \binom{K}{2} \right) \ge \frac{ (\gamma - \b)^2 \log n }{6 n}, \]
and the probability that \(K\) remains independent 
is at most 
\begin{equation*}
\begin{split}
 \left( 1 - \frac{ (\gamma-\b)^2 \log n }{6 n} \right)^{\a \sqrt{\log n} \cdot n^{3/2}}
\le \exp \left\{ - \frac{(\gamma-\b)^2}{6} \a \log^{3/2} n \cdot \sqrt{n} \right\}.
\end{split}
\end{equation*}
On the other hand, the number of \(\gamma \sqrt{n \log n} \)-element sets of vertices is
\[ \binom{n}{ \gamma \sqrt{n \log n} } \le \left( \frac{ n e}{ \g \sqrt{n \log n} } \right)^{ \gamma \sqrt{\log n \cdot n} }
\le \exp \left\{ \frac{\gamma}{2} \log^{3/2} n \cdot \sqrt{n} \right\}. \]
Theorem~\ref{thm:upper} now follows from the union bound.

\section{The \(K_4\)-free Process}
\label{sec:extra}

We prove Theorem~\ref{thm:extra} by
analyzing the 
\(K_4\)-free process on \(n\) vertices, showing that 
it produces a graph with independence number 
\( O ( \log^{4/5}n \cdot n^{2/5} ) \).

As in the analysis of the \( K_3\)-free process, we 
let \( E_i \) be the set of edges chosen through the
first \(i\) steps in the process, \(C_i \subseteq \binom{[n]}{2}\) be the set of
forbidden pairs in \( G_i\) and \( O_i \subseteq \binom{[n]}{2} \) be the set of
available pairs in \( G_i\).

We track the following random variables through the evolution 
of the \(K_4\)-free process.  Let \( Q(i) \) be \( |O_i|\), the number of open pairs in \( \binom{[n]}{2} \) after
\(i\) steps of the process.  
For \( A \in \binom{[n]}{2} \) and \( f \in\{ 0, 1,2,3, 4 \} \) 
let \( X_{A,f}(i) \) be the collection of sets \( B \in \binom{[n]}{2} \) such that
\[ \left|  E_i \cap \binom{ A \cup B}{2} \right| = f  \ \ \ \ \text{ and } \ \ \ \
C_i \cap \binom{ A \cup B}{2} \subseteq \binom{A}{2}.
\]
Furthermore, for \( f \in \{ 0, 1,2,3 \} \) and \( A \in \binom{[n]}{3} \) let 
\( Y_{A,f}(i) \) be the set of vertices \(v\) such that
\[ \left| E_i \cap \left( A \times \{v\} \right) \right| = f   \ \ \ \ \text{ and } \ \ \ \
C_i \cap \binom{ A \cup B}{2} \subseteq \binom{A}{2}. \]
Of course, the random variables \( X_{A,f} \) are the variables we are most interested in tracking;
the variables \( Y_{A,f} \) are introduced in order to maintain bounds on the one-step changes in
the variables that comprise \( X_{A,f} \).
We stop tracking the variables once \( \binom{A}{2} \subseteq E_i \), formally setting
\( X_{A,f}(i) = X_{A,f}(i-1) \) and \( Y_{A,f}(i) = Y_{A,f}(i-1) \) in this situation.
Our scaling is given by \( t= t(i) = i/ n^{8/5} \).

We introduce functions \(q(t)\), \( x_f(t) \) for \( f = 0, 1, 2 ,3, 4 \), and \( y_f(t) \) for \( f = 0,1,2\).  Our 
guess for the purpose of setting up the differential equations is the following
\[
Q(i) \approx q(t) n^2  \hskip1cm
\left| X_{A,f}(i) \right| \approx x_f(t) n^{2 - \frac{2f}{5}}  \hskip1cm
\left| Y_{A,f}(i) \right| \approx y_f(t) n^{1 - \frac{2f}{5}}.
\]
This leads to the system of differential equations
\begin{gather*}
\frac{ dq}{dt} = - x_4 \ \ \ \ \ \ \ \frac{ dx_0}{dt} = - \frac{ 5 x_0 x_4}{ q}  \ \ \ \ \ \ \
\frac{ dx_f}{dt} =  \frac{ (6-f) x_{f-1} }{ q} - \frac{ (5-f) x_f x_4}{q} \ \  \text{ for }
 f = 1, 2, 3, 4 
\end{gather*}
with initial condition \( q(0) = 1/2 \), \( x_0(0) = 1/2 \) and \( x_1(0) = \dots = x_4(0) = 0 \).  
This has solution
\begin{gather*}
q(t) = \frac{1}{2} e^{- 16 t^5} \ \ \ \ \ \ \ \ \ \
x_f(t) = 2^{f-1} \binom{5}{f} t^f e^{ -16 (5 - f) t^5}    \ \ \ \ \text{ for }
 f = 0, 1, 2, 3, 4 
\end{gather*}
With this solution in hand, we turn to \( y_f(t) \).  Here we have the equations
\[ \frac{ dy_0}{dt} = - \frac{3 y_0 x_4}{q} \ \ \ \ \ \ \ \ \ \ \  \frac{ dy_f}{dt} = \frac{ (4-f) y_{f-1}}{q} - \frac{ (3-f) y_f x_4}{q}
\ \ \ \text{ for } f = 1,2 \]
with initial condition \( y_0(0) = 1, y_1(0) = 0 \) and \( y_2(0) = 0 \).
This has solution
\[ y_f(t) = 2^f \binom{3}{f} t^f e^{ -16( 3-f) t^5 }. \]
Note that this suggests that the \(K_4\)-free process terminates 
with \( \Theta\left( n^{8/5} \cdot \log^{1/5}n \right) \) edges.

In order to state our stability results we introduce error functions that slowly decay as
the process evolves.  The polynomial \( p(t) \) has degree 5 and positive coefficients.  We do not
explicitly define this polynomial; it suffices that its coefficients are sufficiently large.   Define
\begin{gather*}
f_q = \begin{cases}  e^{ p(t)} & \text{ if } t \le 1 \\ 
\frac{ e^{ p(t)} }{ t^4 } & \text{ if } t > 1 \end{cases} 
\ \ \ \ \ \ \ \ \ \ \  f_f =  e^{ p(t) - 16 (4-f) t^5 } \ \ \  \text{ for } f = 0, 1, 2, 3, 4\\
 h_f =  e^{ p(t) - 16 (2-f) t^5 } \ \ \  \text{ for } f = 0, 1, 2. 
\end{gather*}
Define \( \cb_i \) to be the event that there exists \( j \le i \) such that
\[  \left| Q(j) -  q( t(j)) n^2 \right| \ge f_q( t(j)) n^{29/15} \]
or there is a set \( A \in \binom{[n]}{2} \) and \( f \in \{ 0,1, \dots, 4\} \) 
such that \( \binom{A}{2} \not\subseteq E_j \) and
\[  \left|  \left| X_{A,f}(j) \right| -  x_f( t(j)) n^{2 - \frac{ 2f}{5}} \right| \ge f_f( t(j)) n^{2 - \frac{ 2f}{5} - \frac{1}{15}} \]
or there is a set \( A \in \binom{[n]}{3} \) and \( f \in \{ 0,1,2\} \) such that  \( \binom{A}{2} \not\subseteq E_j \) and
\[  \left| Y_{A,f}( j) \right| > y_f(t(j)) n^{1 - \frac{2f}{5}} + h_f( t(j)) n^{1 - \frac{ 2f}{5} - \frac{1}{15}} \]
or there is a set \( A \in \binom{[n]}{3} \) such that  \( \binom{A}{2} \not\subseteq E_j \) and
\[  \left| Y_{A,3}(j) \right| > 15. \]
We introduce absolute constants \( \mu,\rho\) and \( \gamma \).  As in our analysis of the \( K_4\)-free
process, \( \mu \) and \( \rho \) are small relative to \( p(t)\) and \( \gamma \) is
large with respect to \( \mu \).  Define \( m = \mu n^{8/5} \log^{1/5} n \).
\begin{theorem} If \(n\) is sufficiently large then
\label{thm:line}
\[ Pr \left( \cb_{ \mu n^{8/5} \log^{1/5} n} \right) \le n^{-1/6}. \]
\end{theorem} 
\begin{theorem}
\label{thm:dark} If \(n\) is sufficiently large then
\[ Pr \left( \alpha \left( G_{ \a n^{8/5} \log^{1/5} n } \right) > \g n^{2/5} \log^{4/5} n \mid 
\overline{{\mathcal B}_{m } } \right) < e^{- n^{1/15} }. \]
\end{theorem}
\noindent
The methods introduced in Sections~\ref{sec:lower}~and~\ref{sec:upper} can be used to
prove Theorems~\ref{thm:line}~and~\ref{thm:dark}.  This is more or less straightforward and is mostly left to
the reader; we conclude this section with the details that do not follow immediately as above.

\begin{proof}[Proof of Theorem~\ref{thm:line}]
There is one significant difference between the triangle-free process and the \(K_4\)-free process that must be
dealt with here.
In the case
of the triangle-free process, there is a one-to-one correspondence between edges closed when
\(e_i\) is added and vertices that are partial with respect to \(e_i\) in \( G_{i-1} \).  The analogous 
correspondence
does not hold for the \( K_4\)-free process:  Since a pair \( \{u,v\} \) that intersects \( e_i\) could
be a subset of \( B \cup e_i \) for many sets \( B \in X_{e_i,4}(i-1) \), there is not a one-to-one correspondence
between pairs \( \{u,v\}  \) closed by the addition of \( e_i \) to the graph and \( X_{e_i,4}(i-1) \).  In order
to overcome this problem we note that, based on simple density considerations, the difference between these two quantities
is bounded by \( n^{4/15} \) in the event \( \overline{ B_m} \).  

Let \( \epsilon \) be a sufficiently small constant
(This constant is chosen so that \( |O_i| \ge n^{2 - \epsilon} \) for all 
\( i \le m \) in the event
\( \overline{ \cb_m } \)).  Set \( k = n^{1/5 + 8 \e} \) and
let \( \cm_i \) be the event that there exists 
\( \{ u,v \} \in \binom{[n]}{2}  \), distinct vertices \( w_1, \dots, w_k \in [n] \setminus \{u,v\} \) 
and distinct  
\( z_1, z_2, \dots , z_{2k} \in [n] \setminus \{u,v, z_1, z_2, \dots, z_{2k} \} \) such that 
\[ \{ u, w_j\} , \{ u, z_{2j-1} \}, \{ u, z_{2j} \} , \{ v, z_{2j-1} \},
\{v, z_{2j}\}, \{ w_j, z_{2j-1}\}, \{ w_j, z_{2j} \} \in E_i  \text{ for }   j =1, \dots, k. \] 
\begin{claim} If \(n\) is sufficiently large
\[ Pr \left( \cm_m  \wedge \overline{ \cb_m } \right) \le  e^{ - n^{1/5}}. \]
\end{claim}
\begin{proof}
\begin{equation*}
\begin{split}
Pr \left( \cm_m \wedge \overline{ \cb_m } \right) & \le n^2 \cdot \binom{n}{k} \cdot n^{2k} \cdot m^{ 7k } \left( \frac{1}{ n^{2- \epsilon}} \right)^{7k} \\
& \le n^2 \left(  \frac{ e n^{3} \cdot ( \mu n^{8/5} \log^{1/5} n )^7}{ k \cdot n^{14 - 7 \epsilon}} \right)^k \\
& = n^2 \left( \frac{ e \mu^7 n^{1/5 + 7 \epsilon} \log^{7/5} n }{k} \right)^k. 
\end{split}
\end{equation*}
\end{proof}
\noindent
Now let \( \ell = n^{6\e} \) and
let \( \cn_i \) be the event that there exist vertices \(u,v,z \) and disjoint 
sets \(A ,B \in \binom{[n]}{ \ell}  \) 
such that \( A \subseteq N_i(u) \cap N_i(z) \) , \( B \subseteq N_i(u) \cap N_i(v) \) and \(G_i\) has
a matching of \( \ell \) edges in \( A \times B \).
\begin{claim}
If \(n\) is sufficiently large then
\[ Pr( \cn_m \wedge \overline{ \cb_m } ) \le e^{- n^{5\e}}.   \]
\end{claim}
\begin{proof}
\begin{equation*}
\begin{split}
Pr \left( \cn_m \right) & \le n^3 \cdot \binom{n}{ \ell } \cdot n^{ \ell } \cdot m^{5 \ell} \left( \frac{1}{ n^{2 - \e}} \right)^{5 \ell} \\
& \le n^3 \left( \frac{ e n^2 \cdot ( \mu n^{8/5} \log^{4/5} n)^5}{ \ell \cdot n^{10-5\e}} \right)^\ell .
\end{split}
\end{equation*}
\end{proof}
\noindent
Now suppose \( \omega_{j-1} \not\subseteq \cb_{j-1} \vee \cm_{j-1} \vee \cn_{j-1} \).  Let \(W\) be the set of vertices 
\(w\) such that \( \{v,w\} \) is closed by the addition of the edge \( e_j = \{u,v\} \) and
there exist distinct vertices \( z_w, z_w^\prime \) such that \( \{w, z_w\}, \{ w, z_w^\prime\} \in X_{ e_j,4}(j-1) \).
Note that
\( \omega_{j-1} \not\subseteq \cb_{j-1} \) implies that
the number of pairs \(B \in X_{e_{j},4}(j-1) \) 
that correspond to a particular vertex \( w \in W \) is at 
most \( 16 \).  Thus the difference
between \( | X_{e_j,4}(j-1) | \) and the number of pairs closed by the 
addition of edges \( e_j \) is at most 
\( 32 |W| \).  It remains to argue that \( W \) is small.  First note that 
\( \omega_{j-1} \not\subseteq \cn_{j-1} \) implies that each vertex \( z \) 
is in the set \( \{ z_w, z_w^\prime \} \) for at most \( 16 n^{6 \e} \) vertices 
\(w \in W \).
Let \( W^\prime \subseteq W \) be a maximum 
set such that \( a,b \in W^\prime \) implies 
\( \{ z_a, z^\prime_a \} \cap \{ z_b , z^\prime_b \} = \emptyset \).  By the previous observation (using 
\( \omega_{j-1} \not\subseteq \cn_{j-1} \) )
we have 
\( |W^\prime| \ge |W|/ (16n^{6\e}) \).
Furthermore, \( \omega_{j-1} \not\subseteq \cm_{j-1} \) implies that \( |W^\prime| < n^{1/5 + 8\e} \).  Thus, the 
number of pairs closed
by the addition of \( e_j \) is in the interval
\[  \left[ \left|  X_{e_j,4}(j-1) \right| - 32 \cdot 16 n^{1/5 + 14\epsilon},  \left|  X_{e_j,4}(j-1) \right| \right], \]
which is sufficient for the proof.
\end{proof}

\begin{proof}[Proof of Theorem~\ref{thm:dark}]
As in the proof of Theorem~\ref{thm:upper}, we fix a set \( K \) of \( \gamma n^{2/5} \log^{4/5}n \) vertices
and show that the probability that \(K\) remains independent is small even when compared with the number of such sets.
We condition on \( \overline{ \cb_m } \) and a bound of \( n^{1/5+ 3\e} \) on all co-degrees.  (This bound on the co-degrees follows from a very simple 
first moment calculation.  We could establish a tighter bound using martingale inequalities, but that is not 
necessary for this argument.)

There are two significant 
differences between the triangle-free process and the \(K_4\)-free process here: the fact that 
in the latter the addition of an
edge \( e_i \) that is disjoint from \(K\) could close many pairs within \(K\) and the fact that
the neighborhood of a single vertex could include \(K\) as a subset.  

We track the number of open pairs within two kinds of subgraphs.
Set \[ k =  \frac{\g}{3} n^{2/5} \log^{4/5} n .\]  
Let \( A, B \in \binom{[n]}{k} \).  We say that 
a pair \( \{u,v\} \in A \times B \)
is {\bf closed with respect to \( A \times B \) } at step \(j\) 
if there exists a step \(i \le j \) such that
\( \{u,v\} \) is among the edges closed by \( e_i = \{x,y\}\) and either
\begin{itemize}
\item[(i)] \( e_i \cap (A \cup B)  = \{y\} = \{u\} \) and there exists \(z \notin A \cup B\) such that
\( \{v,z\} \in X_{\{x,y\},4} \) and \( | N_{i-1}(z) \cap N_{i-1}(x) \cap ( A \cup B)| <  n^{1/5 - \r - 3\e}  \) or
\item[(ii)] \( e_i \cap ( A \cup B ) = \emptyset \) and ( \( | N_{i-1}(e_i) \cap A| \le n^{1/5 - \r - 3\e}  \) or 
\( \ | N_{i-1}(e_i) \cap  B | \le n^{1/5- \r -3\e} \) ).
\end{itemize}
If the pair \( \{u,v\} \in A \times B \) is neither closed with respect to \( \{u,v\} \) nor in the edge 
set \( E_j\) then it is { \bf open with respect to \( A \times B \) }.  Note that, since we assume 
co-degrees are bounded by \( n^{1/5+ 3\e} \), the change in the number of pairs closed with respect to \( A \times B \)
that results from the addition of an edge \( e_i\) is at most \( n^{2/5-\r} \).
It follows from the techniques in
Section~\ref{sec:upper} that with
high probability we have the following: For all steps \( j \le m\) and all pairs \(A,B \in \binom{[n]}{k} \) such that
\( (A \times B) \cap E_j = \emptyset\) the number of edges in \( A \times B \) that are open with respect to \(A \times B \)
is at least \( \frac{k^2}{2} e^{ - 16 t(j)^5 } \). 

We also track the number of open pairs within sets \(D\) consisting of \(k\) vertices.  We say that 
a pair \( \{u,v\} \in \binom{D}{2} \)
is {\bf closed with respect to \( D \) } at step \(j\) 
if there exists a step \(i \le j \) such that
\( \{u,v\} \) is among the edges closed by \( e_i = \{x,y\}\) and either
\begin{itemize}
\item[(i)] \( e_i \cap D  = \{y\} = \{u\} \) and there exists \(z \not\in D\) such that
 \( | N_{i-1}(z) \cap N_{i-1}(x) \cap D| <  n^{1/5 - \r - 3\e}  \) and \( \{v,z\} \in X_{\{x,y\},4} \).
\item[(ii)] \( e_i \cap D = \emptyset \) and \( | N_{i-1}(x) \cap N_{i-1}(y) \cap D| < n^{1/5 - \r/2}  \).
\end{itemize}
If the pair \( \{u,v\} \in  \binom{D}{2} \) is neither closed with respect to \( \{u,v\} \) nor in the edge 
set \( E_j\) then it is { \bf open with respect to \( D \) }.  Again following the techniques in
Section~\ref{sec:upper}, we see that with
high probability we have the following: For all steps \( j \le m\) and all sets \(D \in \binom{[n]}{k} \) such that
\( \binom{D}{2} \cap E_j = \emptyset\) the number of edges in \( \binom{D}{2} \) that are open with respect to \( D \)
is at least \( \frac{k^2}{4} e^{ - 16 t(j)^5 } \).

It remains to show that every set \(K\) of \( \gamma n^{2/5} \log^{4/5} n \)
vertices contains:
\begin{itemize}
\item[(a)] Disjoint sets 
\(A,B\) of \(k\) vertices such the difference between the number of pairs in \(A \times B \) that are open and the number that
are open with respect to \( A \times B\) is less than, say, \( n^{23/30} \), or
\item[(b)] A set \( D \) of \(k\) vertices such the difference between the number of pairs within \( D \) that are open and the number that
are open with respect to \( D\) is less than \( n^{23/30} \).
\end{itemize}
A main tool here is the following observation which follows from a simple first moment calculation.
Let \( \cm_{i} \) be the event that there exist integers \(r,s\) such that \( s \ge n^{2 \epsilon} \), 
\( r\cdot s \ge n^{2/5 + \e} \) and 
disjoint sets \( X \in \binom{[n]}{k} \) and 
\( Y \in \binom{[n]}{ r  } \) such that 
\[ \left| N_m(y) \cap X \right| \ge s  \ \ \ \text{ for all } \ \ \ y \in Y. \]
\begin{claim}
\label{cl:density} If \(n\) is sufficiently large then
\( Pr\left(  \cm_m \right) \le e^{- n^{2/5} } \).
\end{claim}
%
\noindent
Now, let \( L_j \) be the set of vertices \(x \not\in K \) such that
\[ | N_j(x) \cap K | \ge  n^{1/5- \r -3\e} .\]  
Let \( L_j = \{ x_1, x_2, \dots \} \) be arranged in decreasing order of
\( | N_j(x_\ell) \cap K | \).  A simple case analysis 
in conjuction with Claim~\ref{cl:density} now establishes the desired property.  

\vskip1mm

\noindent
{\bf Case 1.} \( |N_j( x_1) \cap K| \ge k \).

\vskip1mm
\noindent
Consider \( D \subseteq N_j(x_1) \cap K \) such that \( |D| = k \).  Note that, appealing to the bound on
common neighbors of triples of vertices given by conditioning on \( \overline{\cb_m} \), all pairs within \(D\) that are closed
but not closed with respect to \(D\) are contained in \( N_j(x_1) \cap N_j(y) \cap D \) where \( y \in L_j\) and 
\( \{x_1,y \} \in E_j \).   The number of such vertices \(y\) is at most \( n^{1/5 + \r + 4\e} \) by Claim~\ref{cl:density}.  
Each such neighborhood includes less than \( n^{2/5 + 6\e} \) edges because of the bound on the co-degrees.  Therefore, the
number of spoiled pairs within \(D\) is at most \( n^{3/5 + \r + 10\e} \).

\vskip1mm

\noindent
{\bf Case 2.}  \( |N_j( x_1) \cap K|  < k \).

\vskip1mm

\noindent
Choose
\[ A \subseteq \bigcup_{\ell=1}^{\ell^\prime} N_j(x_\ell) \cap K \ \ \ \ \ \ \ 
B \subseteq K \setminus \left( \bigcup_{\ell=1}^{\ell^\prime} N_j(x_\ell) \right)  \]
such that \( |A| = |B| =k \)
where \(\ell^\prime \) is the smallest index such that the cardinality of this 
union is at least \(k\).  

First suppose \( \ell^\prime < n^{2/15} \).  
Note that no pairs in \(A \times B \) are spoiled in this case:  If \( x, y \not\in K \)
then either \( N_j(x) \cap N_j(y) \subseteq A \) or \( |N_j(x) \cap N_j(y) \cap A| \le 16 n^{2/15} \)
(using the bound on common neighbors of triples of vertices).

Finally, suppose \(  \ell^\prime \ge n^{2/15} \).  Note that, by Claim~\ref{cl:density}, we
have \( |N_j(x_{\ell^\prime}) \cap K| \le n^{4/15+\e} \) and \( |L_j| \le n^{1/5 + \r + 4\e} \).  Thus, the number of
spoiled pairs is at most  \( n^{8/15 + 2\e} n^{1/5+ \r + 4\e} = n^{11/15 + \r + 6\e} \).

\end{proof}

\section{Martingale Inequalities}
\label{sec:ah}

Lemmas~\ref{lem:sub}~and~\ref{lem:super} follow from the original  
martingale inequality of Hoeffding. 
\begin{theorem}[Hoeffding \cite{hoe}]
\label{thm:hoe}
Let \( 0 \equiv X_0, X_1, \dots \) be a sequence of random variables such that 
\[ X_{k-1} - \mu_k  \le X_k  \le  X_{k-1} + 1- \mu_k \]
for some constant \( 0 < \mu_k < 1 \) for k =1, \dots, m.
Set 
\( \mu = \frac{1}{m} \sum_{k=1}^m \mu_k \) and \( \overline{\mu} = 1- \mu \).
If \( X_0, X_1, \dots \) is a supermartingale and \( 0 < t < \overline{ \mu } \) then
\begin{equation}
\label{eq:hoe}
Pr \left( X_m \ge mt \right) \le \left[  \left[ \frac{ \mu}{ \mu+t} \right]^{\mu +t} 
\left[ \frac{ \overline{ \mu}}{ \overline{\mu} -t} \right]^{ \overline{\mu} -t} \right]^m. 
\end{equation}
\end{theorem}
\noindent
Hoeffding's result was for martingales, but the extension to 
supermartingales is straightforward.   For a survey of applications of this and similar
results, see McDiarmid \cite{mcd}.

In order to apply Theorem~\ref{thm:hoe} 
to the martingales considered in this paper, we introduce the following function.  For \( 0<v <1/2\) set \( \ov = 1-v \) and
define
\[ g(x) = g(x,v) = ( v + x v ) \log \left( \frac{ v}{ v + x v} \right) + ( \ov - x v) \log \left( \frac{ \ov}{ \ov - x v} \right)
\    \ \ \ \text{ for } -1 < x < 1.
\]
Note that, under the conditions of Theorem~\ref{thm:hoe}, we have
\[  Pr (X_m \ge m x \mu) \le e^{ g(x,\mu) m } \ \ \ \ \ \  \text{ and } \ \ \ \ \ \ 
 Pr (X_m \ge m x \om) \le e^{ g(-x,\om) m }. \]
Note further 
\[  g^{\prime \prime}(x) = - \frac{ v}{ (1+x)(\ov - x v) }.\]

\begin{proof}[Proof of Lemma~\ref{lem:sub}]
Let \( 0 \equiv A_0, A_1, \dots \) be a \( (\eta, N) \)-bounded submartingale with \(N \ge 2 \eta\). Let \( a \le m \eta \).  Define
\( X_i = -A_i/( \eta + N) \).  Note that
Theorem~\ref{thm:hoe}  applies to 
\( X_0, X_1, \dots \) with \( \mu = N/( \eta + N) \).  Thus
\[
Pr( A_m \le -a) = Pr\left( X_m \ge \frac{a}{ \eta + N } \right) \le \exp\left\{g\left( - \frac{a}{m \eta}, \om  \right) m \right\}.
\]
It remains to bound \( g(x) \).  Note that if \( -1 < x \le 0 \) then \(  g^{\prime \prime}(x) \le - v \).  
As \( g(0)=g^\prime(0)=0 \), it follows that \( g(x) \le - v x^2/2 \) for \( -1 \le x \le 0 \).  Therefore,
\[ Pr( A_m \le -a) \le \exp \left\{ - \frac{ \eta }{ N + \eta } \frac{ a^2 m}{2 m^2 \eta^2} \right\}
\le \exp \left\{ - \frac{ a^2 }{2 m \eta(N+ \eta)} \right\}. \]
\end{proof}

\begin{proof}[Proof of Lemma~\ref{lem:super}]
Let \( 0 \equiv A_0, A_1, \dots \) be a \( (\eta, N) \)-bounded supermartingale with \( N \ge 10 \eta\).  
Let \( a < \eta m \).
Define \( X_i = A_i/( \eta + N) \).   Theorem~\ref{thm:hoe} applies to 
\( X_0, X_1, \dots \), with \( \mu = \eta/( \eta + N) \).  We have
\[ Pr( A_m \ge a) = Pr\left( X_m \ge \frac{a}{ \eta + N } \right) \le \exp\left\{g\left(  \frac{a}{m \eta}, \mu  \right) m \right\}.  \]
It remains to bound \( g(x) \).  Note that for \( x \ge 0 \) we have
\( g^{\prime \prime}(x) \le -v /(1+x) \).
Since \( g(0)=g^\prime(0)=0 \), this implies
\begin{gather*}
g(x) \le - v \left[ (1+x)\log(1+x) -x \right]  
\le v \left[ -\frac{x^2}{2} + \frac{ x^3}{6} - \frac{ x^4}{12} + \frac{ x^5}{20} \right] \le -\frac{11}{30} v x^2.  
\end{gather*}
Thus
\[
 Pr( A_m \ge a)
\le \exp \left\{ - \frac{11}{30} \frac{ a^2}{ m \eta (N + \eta)}   \right\} 
\le \exp \left\{ - \frac{ a^2}{ 3 m \eta  N}   \right\}.
\]
\end{proof}

\vskip5mm

\noindent
{\bf Acknowledgement.}  The author thanks Alan Frieze, Peter Keevash and Joel Spencer for
many useful comments on earlier drafts of this paper.


\begin{thebibliography}{99}

\bibitem{aks} M. Ajtai, J. Koml\'os and E. Szemer\'edi, A note on Ramsey numbers. {\em Journal of
Combinatorial Theory A} {\bf 29} (1980) 354-360.

\bibitem{a} N. Alon, Explicit ramsey graphs and orthonormal labelings. {\em Electronic Journal of
Combinatorics} {\bf 1} (1994) R12, 8pp.


\bibitem{bbfp} A. Beveridge, T. Bohman, A. Frieze and O. Pikhurko, 
Product rule wins a competitive game. {\em Proceedings of the American Mathematical Society} 
{\bf 135} (2007) 3061-3071. 

\bibitem{bf} T. Bohman and A. Frieze, Karp-Sipser on random graphs with a fixed 
degree sequence, submitted. 

\bibitem{bk} T. Bohman and D. Kravitz, Creating a giant component.  
{\em Combinatorics, Probability and Computing} {\bf 15} (2006) 489-511. 


\bibitem{bela!} B. Bollob\'as, personal communication.

\bibitem{br} B. Bollob\'as and O. Riordan, Constrained graph processes, 
{\em Electronic Journal of Combinatorics} {\bf 7(1)} (2000) R18.


\bibitem{cg} F. Chung and R. Graham, {\em Erd\H{o}s on Graphs: His Legacy 
of Unsolved Problems}. A.K. Peters 1999.


\bibitem{e1} P. Erd\H{o}s, Graph theory and probability II. {\em Canadian Journal of Mathematics}
{\bf 13} (1961) 346-352.

\bibitem{el} P. Erd\H{o}s and L. Lov\'asz, Problems and results on 3-chromatic hypergraphs
and some related questions, in {\em Infinite and Finite Sets}. North Holland 1975.

\bibitem{esw} P. Erd\H{o}s, S. Suen and P. Winkler, On the size of a random maximal graph.
{\em Random Structures and Algorithms} {\bf 6} (1995) 309-318.


\bibitem{grs} R. Graham, B. Rothschild and J. Spencer, {\em Ramsey Theory}. Wiley, 1990.

\bibitem{e2} J. Graver and J. Yackel, Some graph theoretic results associated with
Ramsey's theorem. {\em Journal of Combinatorial Theory} {\bf 4} (1968) 125-175.

\bibitem{hoe} W. Hoeffding, Probability inequalities for sums of bounded variables. 
{\em Journal of the American Statistical Association} {\bf 58} (1963) 13-30.

\bibitem{pete} P. Keevash, personal communication.

\bibitem{kim} J. H. Kim, The ramsey number \( R(3,t) \) has order of magnitude \( t^2/ \log t\).
{\em Random Structures and Algorithms} {\bf 7} (1995) 173-207.


\bibitem{e5} M. Krivelevich, Bounding Ramsey numbers through large deviation inequalities,
{\em Random Structures and Algorithms} {\bf 7} (1995) 145-155.  


\bibitem{ot} D. Osthus and A. Taraz, Random maximal \(H\)-free greaphs. {\em Random Structures
and Algorithms} {\bf 18} (2001) 61-82.

\bibitem{mcd} C. McDiarmid, On the method of bounded differences, in 
{\em Surveys in Combinatorics, London Mathematics Society Lecture Note Series 141.}
Cambridge University Press 1989.

\bibitem{rw} A. Ruci\'nski and N. Wormald, Random graph processes with degree restrictions.
{\em Combinatorics, Probability and Computing} {\bf 1} (1992) 169-180.

%
%

\bibitem{e4} J. Shearer, A note on the independence number of triangle-free graphs II.
{\em Journal of Combinatorial Theory B} {\bf 53} (1991) 300-307.

\bibitem{e3} J. Spencer, Asymptotic lower bounds for Ramsey functions. {\em Discrete Mathematics}
{\bf 20} (1977) 69-76.

\bibitem{s} J. Spencer, Maximal triangle-free graphs and Ramsey \( R(3,t) \), 
unpublished manuscript, 1995.

\bibitem{sw} J. Spencer and N. Wormald, Birth control for giants. 
{\em Combinatorica}, to appear.

\bibitem{w} N. Wormald, The differential equations method for random graph processes and greedy algorithms. 
pp. 73-155 in {\em Lectures on Approximation and Randomized Algorithms}, Karonski and Pr\"omel eds. 
PWN, Warsaw 1999.

\end{thebibliography}
\end{document}